\documentclass[12pt]{amsart}
\usepackage{amsmath,amsfonts,amsthm}
\usepackage{amssymb,latexsym}
\usepackage{cite}
\usepackage{cases}
\usepackage[numbers]{natbib} 
\usepackage{enumitem}
\usepackage{hyperref} 
\newtheorem{thm}{Theorem}[section]

\newtheorem{lem}[thm]{Lemma}
\newtheorem{prop}[thm]{Proposition}

\theoremstyle{remark}
\newtheorem{re}[thm]{Remark}

\newcommand{\sgn}{\operatorname{sgn}}
\newcommand{\lcm}{\mbox{lcm}}

\newcommand{\cls}{\mbox{cls}\,}
\newcommand{\spn}{\mbox{spn}\,}
\newcommand{\gen}{\mbox{gen}\,}
\newcommand{\Span}{\mbox{Span}\,}
\newcommand{\R}{\mathbb{R}}
\newcommand{\Q}{\mathbb{Q}}
\newcommand{\C}{\mathbb{C}}
\newcommand{\Z}{\mathbb{Z}}
\newcommand{\N}{\mathbb{N}}

\renewcommand{\Re}{\operatorname{Re}}
\newcommand{\im}{\operatorname{Im}}
\newcommand{\SL}{\operatorname{SL}}
\DeclareSymbolFont{largesymbol}{OMX}{yhex}{m}{n}
\DeclareMathAccent{\Widehat}{\mathord}{largesymbol}{"62}
\numberwithin{equation}{section}

\title[Sign changes over split or inert primes]{Sign changes of Fourier coefficients of cusp forms of half-integral weight over split and inert primes in quadratic number fields}
\author{Zilong He}
\address{Department of Mathematics, University of Hong Kong, Pokfulam, Hong Kong}
\email{zilonghe@hku.hk}
\author{Ben Kane}
\address{Department of Mathematics, University of Hong Kong, Pokfulam, Hong Kong}
\email{bkane@hku.hk}
\date{\today}
\begin{document}
\begin{abstract}
In this paper, we investigate sign changes of Fourier coefficients of half-integral weight cusp forms. In a fixed square class $t\Z^2$, we investigate the sign changes in the $tp^2$-th coefficient as $p$ runs through the split or inert primes over the ring of integers in a quadratic extension of the rationals. We show that infinitely many sign changes occur in both sets of primes when there exists a prime dividing the discriminant of the field which does not divide the level of the cusp form and find an explicit condition that determines whether sign changes occur when every prime dividing the discriminant also divides the level.
\end{abstract}
\keywords{half-integral weight modular forms, sign changes, Fourier coefficients, quadratic number fields, quadratic forms}
\subjclass[2010]{11F37,11F30,11N69,11R11,11E20}

\thanks{The research of the second author was supported by grants from the Research Grants Council of the Hong Kong SAR, China (project numbers HKU 17302515, 17316416, 17301317 and 17303618).}
\maketitle

\section{Introduction}
Throughout this paper, we let $ k\ge 1 $ and $ N\ge 4 $ be integers and $ 4\mid N $. We denote by $ S_{k+1/2}(N,\psi) $ the space of cusp forms of weight $ k+1/2 $ for the group $ \Gamma_{0}(N) $ with a Dirichlet character $ \psi $ modulo $ N $ and $ S_{3/2}^{*}(N,\psi) $ the orthogonal complement with respect to the Petersson scalar product of the subspace $ U(N,\psi) $ generated by unary theta functions. We also set $ S^{*}_{k+1/2}(N,\psi)=S_{k+1/2}(N,\psi) $ for $ k\ge 2 $. 

Each $ \mathfrak{f}\in S_{k+1/2}^{*}(N,\psi) $ has a Fourier expansion given by
\[
\mathfrak{f}(z)=\sum\limits_{n=1}^{\infty}\mathfrak{a}_{\mathfrak{f}}(n)q^{n},
\]
where $ q:=e^{2\pi iz} $ with $ z\in\mathbb{H} $, the complex upper half-plane. Under the assumption that $\mathfrak{a}_{\mathfrak{f}}(n)\in\R$, many authors have studied the change of signs $\sgn\left(\mathfrak{a}_{\mathfrak{f}}(n)\right)$ as $n$ runs through natural sequences (see for instance \cite{BruinierKohnen,IwaniecKohnenSengupta,KnoppKohnenPribitkin,KohnenSengupta,LauWu}).  For example, in \cite[Theorem 1]{kohnen_fourier_2013} it was shown that for a squarefree positive integer $ t $, if $\mathfrak{a}_{\mathfrak{f}}(t)\neq 0$ for some $\mathfrak{f}\in S_{k+1/2}^*(N,\psi)$, then there are infinitely many sign changes in the sequence
\begin{equation}\label{eqn:primeseq}
\left(\mathfrak{a}_{\mathfrak{f}}\left(tp_n^2\right)\right)_{n=1}^{\infty},
\end{equation}
where $p_n$ is the $n$-th prime. Formally, given a real sequence $ \{a(n)\}_{n=1}^{\infty} $, we say that $ \{a(n)\}_{n=1}^{\infty} $ \textit{exhibits} or \textit{has a sign change} at $n_0\in\N$ (or between $n_0$ and $n_0+1$) if $ a(n_{0})a(n_{0}+1)<0$. Letting $K:=\mathbb{Q}\left(\sqrt{D}\right)$ be a quadratic extension of $\Q$, where $1\neq D\in \Z$ is a fundamental discriminant, it is natural to ask whether there are infinitely many sign changes when the sequence \eqref{eqn:primeseq} is restricted to the subsequence $\mathfrak{a}_{\mathfrak{f}}\left(tp^2\right)$ with $p$ running over all split (resp. inert) primes in the ring of integers $\mathcal{O}_{K}$, or even more generally in arithmetic progressions $p\equiv m\pmod{M}$ for some fixed $m$ and $M$.  Since $p$ is split (resp. inert) in $\mathcal{O}_K$ if and only if $\chi_D(p)=1$ (resp. $\chi_D(p)=-1$), where $\chi_D(n):=(D/n)$ denotes the Kronecker--Jacobi--Legendre symbol, we let $p_{D,n,+}$ denote the $n$-th prime which is split in $\mathcal{O}_K$ and $p_{D,n,-}$ denote the $n$-th prime which is inert in $\mathcal{O}_K$. For each $\varepsilon\in \{\pm \}$, we investigate sign changes across the sequences
\begin{equation}\label{eqn:primeDseq}
\left(\mathfrak{a}_{\mathfrak{f}}\left(tp_{D,n,\varepsilon}^2\right)\right)_{n=1}^{\infty}.
\end{equation}
\begin{thm}\label{thm:signchange}
Let $ k\ge 1 $ be an integer, $ N\ge 4 $ an integer divisible by $ 4 $, and $ \psi $ be a Dirichlet character modulo $ N $. Suppose that $\mathfrak{f}\in S_{k+1/2}^*(N,\psi)$ has real Fourier coefficients and $t\geq 1$ is a squarefree integer such that $\mathfrak{a}_{\mathfrak{f}}(t)\neq 0$. If $D$ is a fundamental discriminant for which there exists an odd prime $\ell\mid D$ with $\ell\nmid N$, then there are infinitely many sign changes in both of the sequences
\[
\left(\mathfrak{a}_{\mathfrak{f}}\left(tp_{D,n,+}^2\right)\right)_{n=1}^{\infty}
\]
and 
\[
\left(\mathfrak{a}_{\mathfrak{f}}\left(tp_{D,n,-}^2\right)\right)_{n=1}^{\infty}.
\]
More specifically, there exists a small constant $ \delta=\delta_{\mathfrak{f},t,D}>0 $ such that for sufficiently large $x$, there is a sign change with $p_{D,n,\varepsilon}$ in the interval $[x^{\delta},x]$.
\end{thm}


In order to describe the existence or non-existence of sign changes when every odd prime dividing $D$ also divides the level of the cusp form $\mathfrak{f}\in S_{k+1/2}^*(N,\psi)$, we require the Shimura lift \cite{Shimura1973}. For a squarefree positive integer $ t $, $ \mathfrak{f} $ can be lifted to a cusp form $ f_{t}\in S_{2k}(N/2,\psi^{2}) $ 
\[
f_{t}(z)=\sum\limits_{n=1}^{\infty}a_{f_{t}}(n)q^{n}
\]
by the $t$-th Shimura correspondance. Here the $ n $-th coefficient $ a_{f_{t}}(n)$ of $ f_{t} $ is given by
\begin{equation}\label{eq:relationaft_af}
a_{f_{t}}(n)=\sum\limits_{d\mid n}\psi_{t,N}(d)d^{k-1}\mathfrak{a}_{\mathfrak{f}}(t\dfrac{n^{2}}{d^{2}}),
\end{equation}
where $ \psi_{t,N} $ denotes the character
\[
\psi_{t,N}(d):=\psi(d)\left(\dfrac{(-1)^{k}t}{d}\right).
\]
Thus in particular 
\begin{equation}\label{eq:Shimuraliftcoeff}
a_{f_t}(p)=\psi_{t,N}(p) p^{k-1}\mathfrak{a}_{\mathfrak{f}}(t)+\mathfrak{a}_{\mathfrak{f}}\left(tp^2\right).
\end{equation}
We also require some properties of quadratic twists of modular forms. For a modular form $ f $ with Fourier expansion ($q:=e^{2\pi i z}$)
\[
f(z)=\sum_{n\geq 1}a_f(n) q^n
\]
and a character $\chi$, define \begin{it}the twist of $f$ by $\chi$\end{it} by
\begin{equation}\label{eqn:ftwist}
(f\otimes \chi)(z):=\sum_{n\geq 1} \chi(n)a_f(n) q^n.
\end{equation}
If $f$ is a primitive form, then $f\otimes\chi$ is a Hecke eigenform, but not necessarily primitive. We write $f_{\chi}^{\star}$ for the primitive form associated to the twist $f\otimes \chi$ (see Remark \ref{re:inducedprimitiveforms} for details). Also, we say that two characters $ \chi_{1} $ and $ \chi_{2} $ are \textit{almost equal} if $ \chi_{1}(p)=\chi_{2}(p) $ for all primes $ p $ not dividing their smallest positive periods and denote it by $ \chi_{1}\simeq \chi_{2} $. Otherwise, we say that $ \chi_{1} $ and $ \chi_{2} $ are not almost equal and denote it by $ \chi_{1}\not\simeq\chi_{2} $.

\begin{thm}\label{thm:signchangedivisor}
Let $ k\ge 1 $ be an integer, $ N\ge 4 $ an integer divisible by $ 4 $, and $ \psi $ be a Dirichlet character modulo $ N $. Suppose that $\mathfrak{f}\in S_{k+1/2}^*(N,\psi)$ has real Fourier coefficients and $t\geq 1$ is a squarefree integer such that $\mathfrak{a}_{\mathfrak{f}}(t)\neq 0$. Let $D\neq 1$ be a fundamental discriminant for which every odd prime dividing $D$ also divides $N$. Then the following hold.
\begin{enumerate}[leftmargin=*,align=left,label={\rm(\arabic*)}]
\item The sequence \eqref{eqn:primeDseq} (with $\varepsilon=\pm$ fixed) restricted to the primes $p\nmid N$ does not exhibit sign changes if and only if $\psi_{t,N}\simeq\chi_{D}^{j}$ with $ j\in\{0,1\} $ and the $t$-th Shimura correspondence $ f_{t} $ satisfies 
\begin{equation}\label{eqn:ftnosignchange}
f_t\otimes\chi_{N^2}= \sum_{i} c_i\left(f_i-\varepsilon \left(f_i\right)_{\chi_D}^{\star}\right)\otimes\chi_{N^2}
\end{equation}
for some $c_i\in\C$ and where $f_i$ runs through a full set of primitive forms of level $M\mid N$. Moreover, if \eqref{eqn:primeDseq} does not exhibit sign changes for $\varepsilon$, then it does exhibit sign changes for $-\varepsilon$. There is at most one squarefree $t$ for which no sign changes occur and, if it exists, $t\mid N$. 

\item
 There exists a choice of $N$, $t\geq 1$ squarefree, a Dirichlet character $\psi$ modulo $N$, a fundamental discriminant $D\neq 1$, and $\mathfrak{g}\in S_{3/2}^*(N,\chi)$ such that $\mathfrak{a}_{\mathfrak{g}}\left(tp_{D,n,\varepsilon}^2\right)$ exhibits sign changes for precisely one of $\varepsilon=\pm$.
\end{enumerate}

\end{thm}
\begin{re}
The twist by $\chi_{N^2}$ in Theorem \ref{thm:signchangedivisor} (1) precisely annihilates the coefficients that are not relatively prime to the level, which do not affect the sign changes except for possibly finitely many primes dividing $N$ and not dividing $D$.

If $f_i$ and $f_i^{\star}=f_j$ are both newforms of level dividing $N$, then the term $f_i-\varepsilon f_i^{\star}$ occurs twice (once for $i$ and once for $j$) unless $f_i=f_i^{\star}$, in which case we say that $f_i$ has CM by $\chi_D$ (see \eqref{eqn:CMdef}).

Note that although Theorem \ref{thm:signchangedivisor} (1) gives an if and only if statement, it is not immediately clear that the conditions are consistent with the assumption that $\mathfrak{a}_{\mathfrak{f}}(t)\neq 0$. The existence of such a form is the content of Theorem \ref{thm:signchangedivisor} (2). The counterexample from Theorem \ref{thm:signchangedivisor} (2) is constructed via the theory of quadratic forms, and in particular spinor genus theory. Although we only construct one explicit example, many can be constructed in an analogous way. 

We determine a precise criterion which implies alternation of the coefficients in both cases and obtain Theorem \ref{thm:signchange} by showing that the criterion cannot be satisfied in this case; to obtain Theorem \ref{thm:signchangedivisor} (1) we need to determine precisely when this criterion holds. Arguing via orthogonality of characters, one should be able to generalize the results in this paper to show that there are sign changes in arithmetic progressions $p\equiv m\pmod{M}$ as long as $\gcd(M,N)=1$, but the counterexample in Theorem \ref{thm:signchangedivisor} (2) implies that the gcd condition is necessary.

\end{re}

 It might be interesting to investigate sign changes of Fourier coefficients of integral weight cusp forms across split or inert primes or arithmetic progressions as well. 

The paper is organized as follows. In Section \ref{sec:prelim}, we give some preliminaries and necessary information about quadratic twists. In Section \ref{sec:convolution} we give some useful information about the growth of convolution $L$-functions. In Section \ref{sec:SignChangeThm}, we investigate the case when not every prime dividing $D$ divides the level of the modular form, proving Theorem \ref{thm:signchange}. In Section \ref{sec:signchangedivisor}, we investigate the case when all of the divisors of $D$ divide the level, proving Theorem \ref{thm:signchangedivisor}.

\section*{Acknowledgements}
The authors thank Yuk-Kam Lau for many helpful discussions and the anonymous referees for many useful corrections and comments.

\section{Preliminaries}\label{sec:prelim}
\subsection{Modular forms and quadratic twists}

Let $  \gamma=
\begin{pmatrix}
	a & b \\
	c & d
\end{pmatrix}\in \SL_{2}(\mathbb{Z})$ and $ z\in \mathbb{H} $. A \begin{it}fractional linear transformation\end{it} is defined by 
$\gamma z:=\dfrac{az+b}{cz+d}$.
Write $ j(\gamma,z):=cz+d $. For a finite-index subgroup $ \Gamma\subseteq SL_{2}(\mathbb{Z}) $ and a \begin{it}weight\end{it} $ w\in\mathbb{R} $, a \textit{multiplier system} (we may omit the multiplier if it is trivial) is defined as a function $ \nu:\Gamma\to\mathbb{C} $ such that
\begin{align*}
\nu(\gamma A)j(\gamma A,z)^{w}=\nu(A)j(A,\gamma z)^{w}\nu(\gamma)j(\gamma,z)^{w}
\end{align*}
for all $ \gamma,A\in \Gamma $ and $|\nu(\gamma)|=1$. Here and throughout we take the \begin{it}principal branch\end{it} $Z^w:=|Z|^w e^{iw\operatorname{Arg}(Z)}$ with $-\pi< \operatorname{Arg}(Z)\leq \pi$. Also, the \textit{slash operator} $ |_{w,\nu} $ of weight $ w $ and multiplier system $ \nu $ is defined as
\begin{align*}
f|_{w,\nu}\gamma:=\nu(\gamma)^{-1}j(\gamma,z)^{-w}f(\gamma z).
\end{align*}
We call elements of $\Gamma\backslash\Q\cup\{i\infty\}$ the \begin{it}cusps\end{it} of $\Gamma$. For any cusp $\varrho$ of $\Gamma$, we let $\sigma=\sigma_{\varrho}\in\Q\cup\{i\infty\}$ be a representative and choose a matrix $\gamma_{\sigma}$ that sends $i\infty$ to $\sigma$. We say that $f$ \begin{it}grows at most polynomially\end{it} towards the cusp $\varrho$ (or $\sigma$) if there exists $r\in\R$ such that 
\[
\im(y)^rf\big|_{w}\gamma_{\sigma}(z)
\]
is bounded as $y\to \infty$. 

Using the above notation, we may give a general definition of modular forms that includes both integral and half-integral weight. A \textit{holomorphic modular form} of weight $ w\in\mathbb{R} $ and multiplier system $ \nu $ for the subgroup $ \Gamma $ is a function $ f:\mathbb{H}\to \mathbb{C} $ satisfying the following:
\begin{enumerate}[leftmargin=*,align=left,label={\rm(\arabic*)}]
\item
 $ f(z) $ is holomorphic on $ \mathbb{H} $;
\item
 $ f|_{w,\nu}\gamma=f$ for all $ \gamma\in \Gamma $;
\item $f$ grows at most polynomially towards every cusp.
\end{enumerate}

\noindent If moreover $ f $ vanishes at every cusp, then $ f $ is called a \begin{it}cusp form\end{it}.

We are particularly interested in the case when $w$ is a half-integer and $\Gamma=\Gamma_0(L)$, where 
\[
\Gamma_0(L):=\left\{\begin{pmatrix}
	a & b \\
	c & d
\end{pmatrix}\in \SL_{2}(\mathbb{Z}): L\mid c\right\}.
\]
The \begin{it}theta multiplier\end{it} is defined by 
\[
\nu_{\Theta}:=\frac{\Theta(\gamma z)}{j(\gamma,z)^{\frac{1}{2}} \Theta(z)},
\]
where $\Theta$ is the usual weight $1/2$ unary theta function 
\[
\Theta(z):=\sum_{n\in\Z} q^{n^2}.
\]
For a character $\psi:\Z\to \C$ and $w\in\frac{1}{2}\Z$, we define the multiplier system 
\[
\nu_{\psi,w}(\gamma):=\psi(d) \nu_{\Theta}^{2w}
\]
and call any modular form of weight $w$ and multiplier system $\nu_{\psi,w}$ on $\Gamma=\Gamma_0(L)$ a modular form of weight $w$ with \begin{it}level\end{it} $L$ and \begin{it}Nebentypus character\end{it} (or just character) $\psi$. Let $ S_{k}(L,\psi) $ be the space of holomorphic cusp forms of integral weight $ k\ge 2 $ and level $ L $ and Nebentypus $ \psi $. We denote by $ L^{*} $ the conductor of the character $ \psi $ and by $ S_{k}^{new}(L,\psi) $ the orthogonal complement with respect to the Petersson inner product of the subspace generated by all forms $ g(\ell z) $, where $ g\in S_{k}(M,\psi_{M}) $ has a strictly lower level $M\mid L  $ and $ L^{*}\mid M $. If $ f\in S_{k}^{new}(L,\psi) $ is a common eigenfunction for all Hecke operators and its first coefficient equals one, then $ f $ is called a \begin{it}primitive form\end{it}. We denote by $ H_{k}^{*}(L,\psi) $ the set of all primitive forms of weight $ k $, level $ L $ and Nebentypus $ \psi $.


We require some properties of quadratic twists of modular forms. For a real character $\chi=\chi_D$, we say that a Hecke eigenform $f$ has \begin{it}CM by $\chi$\end{it} (or \begin{it}CM by the field $\Q(\sqrt{D})$\end{it}) if $a_f(p)=0$ whenever $p$ is a prime for which $\chi(p)=-1$, or in other words if 
\begin{equation}\label{eqn:CMdef}
f\otimes \chi=f. 
\end{equation}

The following is well known, but we supply a proof for the convenience of the reader. 
\begin{prop}\label{prop:fchiHecke}
	Let $ \chi $ be a Dirichlet character modulo $ q $. If $ f\in S_{k}(M,\psi) $ is a Hecke eigenform, then $ f\otimes \chi\in S_{k}(Mq^{2},\psi\chi^{2}) $ is also a Hecke eigenform.
\end{prop}
\begin{proof}
	Write $ a_{f}(n) $ for the $ n $-th Fourier coefficient of $ f $ and $ T_{m} $ for the $m$-th Hecke operator as usual. If $ f\in S_{k}(M,\psi) $, then $ f\otimes \chi\in S_{k}(Mq^{2},\psi\chi^{2})$ by \cite[Proposition 17 (b), p.\hskip 0.1cm 127]{koblitz_introduction_1993}. Since $ f $ is a Hecke eigenform, we have
	\begin{align*}
	a_{f}(p^{n+1})=a_{f}(p)a_{f}(p^{n})-\psi(p)p^{k-1}a_{f}(p^{n-1})
	\end{align*}
	for all primes $ p $ and $ n\ge 1 $. Also, $ a_{f}(d_{1}d_{2})=a_{f}(d_{1})a_{f}(d_{2}) $ for any $  d_{1},d_{2}\in \mathbb{N} $ with $ \gcd(d_{1},d_{2})=1 $, i.e., $ a_{f}(n) $ is multiplicative. Combining this with the fact that $ \psi(d)d^{k-1} $ is completely multiplicative, we furthermore have
	\begin{align*}
	\sum\limits_{d\mid\gcd(n,m)} \psi (d) d^{k-1}a_{f}\left(\dfrac{nm}{d^{2}}\right)=a_{f}(m)a_{f}(n)
	\end{align*}
	by \cite[Exercises 30, 31, pp.\hskip 0.1cm 49--50]{apostol_introduction_1986}. Hence one can check that
	\begin{align*}
	a_{T_{m}(f\otimes \chi)}(n)&=\sum\limits_{d\mid\gcd(n,m)} (\psi\chi^{2})(d) d^{k-1}a_{f}(\dfrac{nm}{d^{2}})\chi(\dfrac{nm}{d^{2}})\\
	&=\chi(mn) \sum\limits_{d\mid\gcd(n,m)} \psi (d) d^{k-1}a_{f}(\dfrac{nm}{d^{2}})\\
	&=\chi(m)\chi(n)a_{f}(m)a_{f}(n)\\
	&=\chi(m)a_{f}(m)a_{f\otimes \chi}(n) 
	\end{align*}
	for all $ m\in\mathbb{N} $.
\end{proof}
\begin{re}\label{re:inducedprimitiveforms}
	Given a Dirichlet character $ \chi $ modulo $ q $, if $ f\in S_{k}(M,\psi) $ is a primitive form, then $ f\otimes \chi\in S_{k}(Mq^{2},\psi\chi^{2}) $ is a Hecke eigenform by Proposition \ref{prop:fchiHecke}. In general, $ f\otimes \chi $ may not be primitive, but there exists a unique primitive form $ f^{\star}\in S_{k}(M^{\prime},\psi^{\star}) $ with $ M^{\prime}\mid Mq^{2} $ such that $ \psi\chi^{2}(n)=\psi^{\star}(n) $ and $ a_{f\otimes \chi}(n)=a_{f^{\star}}(n) $ for any $ n $ prime to $ Mq^{2} $. \cite[REMARK (2), p.\hskip 0.1cm 133 and EXERCISE 5, p.\hskip 0.1cm 376]{iwaniec_analytic_2004}.
\end{re}
Given a Dirichlet character $ \chi $ and a Hecke eigenform $ f $, we denote by $ f_{\chi}^{\star} $ (or simply $ f^{\star} $, when the context is clear) the primitive form induced by a cusp form $ f\otimes\chi  $ as in Remark \ref{re:inducedprimitiveforms}. In particular, if $ f\otimes \chi $ is primitive, then $ f_{\chi}^{\star}=f\otimes \chi $. When $ \chi $ is real, $ (f^{\star})^{\star}=f $. 

Although $f\otimes \chi$ need not be primitive, it is primitive under certain conditions. Specifically, from \cite[Proposition 14.19 and 14.20]{iwaniec_analytic_2004}, if $ f\in H_k^{*}(M,\psi) $ and $ \chi $ is a primitive character modulo $ q $ with $ \gcd(q,M)=1 $, then $ f\otimes \chi\in H_{k}^{*}(Mq^{2},\psi\chi^{2}) $. 

\begin{lem}\label{lem:f*}
Let $k,M\in\N$ and $\psi$ and $\chi$ be characters such that the conductor of $\psi$ divides $M$ and there exists a prime $\ell$ such that $\ell$ divides the conductor of $\chi$ but $\ell$ does not divide $M$. Then for any $f\in H_{k}^*(M,\psi)$, we have that $f_{\chi}^{\star}\in H_{k}^{*}\left(M'\ell^2,\psi\chi^2\right)$ for some $M'\in\N$.

 In particular, if $\chi$ is real, $f$ does not have CM by $\chi$ and if $g\in H_{k}^*(N,\psi')$ with $\ell^2\nmid N$ or $\psi'\neq \psi$, then $g\neq f^{\star}$. 
\end{lem}
\begin{proof}
Note that if $\chi=\chi_1\chi_2$, then 
\begin{equation}\label{eqn:f*rel}
f_{\chi}^{\star}=f_{\chi_1\chi_2}^{\star}=\left(f_{\chi_1}^{\star}\right)_{\chi_2}^{\star}.
\end{equation}
Let $\ell \mid q$ with $\ell \nmid M$ be given and split $\chi=\chi_1\chi_2$ so that the conductor $q'$ of $\chi_1$ is relatively prime to $\ell$ and the conductor of $\chi_2$ is an $\ell$-power $\ell^r$. Note that since $h:=f_{\chi_1}^{\star}$ is primitive and its level $M_{h}$ (dividing $M q'^2$) is relatively prime to $\ell$, 
\[
f_{\chi_1}^{\star}\otimes \chi_2
\]
is primitive of level $M_{h}\ell^{2r}$. Since $r\geq 1$, we see by \eqref{eqn:f*rel} that
\[
f_{\chi}^{\star}=\left(f_{\chi_1}^{\star}\right)_{\chi_2}^{\star}=f_{\chi_1}^{\star}\otimes \chi_2
\]
has $\ell^2$ dividing its level. This is the first claim.

Since $\ell^2\nmid M$, we immediately obtain that $f\neq f^{\star}$, so $f$ does not have CM by $\chi$. Finally, if $g\in H_{k}^{*}(N,\psi')$ and $\ell^2\nmid N$, then we immediately obtain that $g\neq f^{\star}$, as they have different levels. If $\psi'\neq \psi$, then they are not equal because they have different Nebentypus. 
\end{proof}

\subsection{Quadratic forms}

Let $ V $ be a quadratic space over $ \mathbb{Q} $ associated with a symmetric bilinear map $ B:V\times V\to \mathbb{Q} $ and write $ Q(x)=B(x,x) $, $ x\in V $. We denote by $ O(V) $ the orthogonal group of $ V $ and $ O^{\prime}(V) $ the kernel of the homomorphism $ \theta:O(V)\to \mathbb{Q}^{\times}/\mathbb{Q}^{\times 2} $ as usual. Let $O_{\mathbb{A}}(V)   $ and $ O_{\mathbb{A}}^{\prime}(V)  $ be the adelic groups of $ O(V) $ and $ O^{\prime}(V) $, respectively. Let $ L $ be a $ \mathbb{Z} $-lattice on $ V $. We define the \textit{class} $ \cls(L) $, \textit{spinor genus} $ \spn(L) $ and \textit{genus} $ \gen(L) $ of $ L $ by the orbits of $ L $ under the actions of $ O(V) $, $ O(V)O_{\mathbb{A}}^{\prime}(V) $ and $ O_{\mathbb{A}}(V) $ respectively [see \cite{kitaoka_arithmetic_1993} for more details].

For $ n\in \mathbb{N} $, if there exists some $ x_{0}\in L $ such that $ Q(x_{0})=n $, then we say that $ n $ is represented by $ L $ and denote by $ r(n,L) $ the number of representation of $ n $ by $ L $. Also, we define the number of representations of $ n $ by the genus (resp. spinor genus) of $ L $ by the \begin{it}Siegel--Weil average\end{it}
\begin{align*}
r(n,\gen(L)):=&\left(\sum\limits_{\tiny K\in \gen(L)}\dfrac{1}{|O(K)|}\right)^{-1}\sum\limits_{\tiny K\in\gen(L)}\dfrac{r(n,K)}{|O(K)|}
\intertext{and}
 r(n,\spn(L)):=&\left(\sum\limits_{\tiny K\in \spn(L)}\dfrac{1}{|O(K)|}\right)^{-1}\sum\limits_{\tiny K\in\spn(L)}\dfrac{r(n,K)}{|O(K)|},
 \end{align*}
where the summuation is over a complete set of representatives of the classes in the genus (resp. spinor genus) of $ L $. A quadratic form $ Q $ can be always associated with a lattice $ L_{Q} $ and hence we abuse the notations $ r(n,Q) $, $ r(n,\gen(Q)) $ and $ r(n,\spn(Q)) $ standing for $ r(n,L_{Q}) $, $ r(n,\gen(L_{Q})) $ and $ r(n,\spn(L_{Q})) $.

For a ternary quadratic form $ Q$ of discriminant $ D $ and level $ N $, the theta series $ \theta_{Q} $ associated with $ Q $ is given by
\[
\theta_{Q}(z):=\sum\limits_{n=0}^{\infty}r(n,Q)q^{n},
\]
and $\theta_{Q}\in M_{3/2}(N,\psi) $ for an appropriate $\psi$  (see for example \cite[Proposition 2.1]{Shimura1973}). It is well known that the theta series can be expressed as
\begin{equation}\label{eqn:thetaQexpand}
\theta_{Q}(z)=E(z)+H(z)+f(z),
\end{equation}
where $ E(z)=\sum\limits_{n=0}^{\infty}a_{E}(n,Q)q^{n} $ is in the space spanned by Eisenstein series, $ H(z)=\sum\limits_{n=0}^{\infty}a_{H}(n,Q)q^{n}\in U(N,\psi) $ and $ f(z)=\sum\limits_{n=0}^{\infty}a_{f}(n,Q)q^{n}\in S_{3/2}^{*}(N,\psi) $ (\cite[Lemma 4]{schulze-pillot_representation_1990}). In the theory of quadratic forms, the coefficients $ a_{E}(n,Q) $ and $ a_{H}(n,Q) $ can be interpreted as (see for example \cite[Theorem 2]{schulze-pillot_representation_1990}),
\begin{equation}\label{eqn:genspnexpand} 
a_{E}(n,Q)=r(n,\gen(Q))\qquad\text{and}\qquad a_{H}(n,Q)=r(n,\spn(Q))-r(n,\gen(Q)).
\end{equation}

\section{Convolution \texorpdfstring{$L$}{L}-series}\label{sec:convolution}

Hereafter, we assume that all the summations involving the notation $ p $ run over the specific subsets of all the primes. The following lemma follows immediately by replacing $ f $ by $f\otimes \chi$ in \cite[Lemma 2.1]{kohnen_fourier_2013} and it agrees with their results when $ \chi $ is a trivial character.

\begin{lem}\label{lem:lambdafg_primitive}
	Let $f\in  H_{k}^{*}(M_{f},\psi_{f})$ and $ g\in H_{k}^{*}(M_{g},\psi_{g})$ whose $ n$-th coefficients are $ \lambda_{f}(n)n^{(k-1)/2} $ and  $ \lambda_{g}(n)n^{(k-1)/2} $. Let $ \chi $ be a primitive character modulo $ q $ and $ \gcd(q,M_{f})=1 $. Then as $ x\to\infty  $, we have
	\begin{equation}\label{eq:lambdaf_primitive}
	\sum\limits_{\substack{p\le x\\p\nmid M_{f}q}}\dfrac{\chi(p)\lambda_{f}(p)}{p}=O(1)
	\end{equation}
	and
	\begin{equation}\label{eq:lambdaf2_primitive}
	\sum\limits_{\substack{p\le x\\p\nmid M_{f}q}}\dfrac{|\lambda_{f}(p)|^{2}}{p}=\log\log x +O(1).
	\end{equation}
	If $ g\not=f\otimes \chi $, then as $ x\to \infty $, we have
	\begin{equation}\label{eq:lambdafg_primitive}
	\sum\limits_{\substack{p\le x\\p\nmid M_{f}q}}\dfrac{\chi(p)\lambda_{f}(p)\overline{\lambda_{g}(p)}}{p}=O(1).
	\end{equation}
	The implied constants in \eqref{eq:lambdaf_primitive} and \eqref{eq:lambdaf2_primitive} depend on the form $ f $ and the character $ \chi $ and that in \eqref{eq:lambdafg_primitive} depends on the forms $ f, g $ and the character $ \chi $.
\end{lem}

Although a primitive form $ f $ twisted with the character $\chi $ may not be primitive in general, we are still able to make use of \cite[Lemma 2.1]{kohnen_fourier_2013} by taking $ f^{\star} $ instead of $ f\otimes \chi $ from Remark \ref{re:inducedprimitiveforms}. 

\begin{lem}\label{lem:lambda_estimation_real}
	Let $f\in  H_{k}^{*}(M_{f},\psi_{f})$ and $ g\in H_{k}^{*}(M_{g},\psi_{g})$ whose $ n$-th coefficients are $ \lambda_{f}(n)n^{(k-1)/2} $ and  $ \lambda_{g}(n)n^{(k-1)/2} $. Let $ \chi$ be a primitive real character modulo $ q $ and $ f^{\star}$ the primitive form induced by $ f\otimes \chi $. 
\begin{enumerate}[leftmargin=*,align=left,label={\rm(\arabic*)}]
\item
As $ x\to \infty $, we have 
\begin{equation}\label{eq:estimation_lambdaf_real}
\begin{aligned}
\sum\limits_{\substack{p\le x\\ p\nmid M_{f}q  \\ \chi(p)=\pm 1}} \dfrac{\lambda_{f}(p)}{p}=O(1).
\end{aligned}
\end{equation}
\item If $f$ does not have CM by $\chi$, then as $x\to\infty$, we have 
 	\begin{align}\label{eq:estimation_lambdaf2_realnonCM}
 	\sum\limits_{\substack{p\le x\\ p\nmid M_{f}q  \\ \chi(p)=\pm 1}} \dfrac{\lambda_{f}(p)\overline{\lambda_{f}(p)}}{p}&=\dfrac{1}{2}\log\log x+O(1),\\
\label{eq:estimation_lambdaffstar_real}
	\sum\limits_{\substack{p\le x\\ p\nmid M_{f}q  \\ \chi(p)=\pm 1}} \dfrac{\lambda_{f}(p)\overline{\lambda_{f^{\star}}(p)}}{p}&=\pm\dfrac{1}{2}\log\log x+O(1).
	\end{align}
If $f$ has CM by $\chi$, then as $ x\to\infty $, we have
\begin{equation}\label{eq:CM}
\sum\limits_{\substack{p\le x\\ p\nmid M_{f}q  \\ \chi(p)=\pm 1}} \dfrac{\lambda_{f}(p)\overline{\lambda_{f}(p)}}{p}=	\sum\limits_{\substack{p\le x\\ p\nmid M_{f}q  \\ \chi(p)=\pm 1}} \dfrac{\lambda_{f}(p)\overline{\lambda_{f^{\star}}(p)}}{p}=\frac{1\pm 1}{2} \log\log x+O(1).
\end{equation}
If $ g\not=f $ and $ g\neq f^{\star} $, then as $ x\to\infty $, we have
		\begin{equation}\label{eq:estimation_lambdafg_real}
	\begin{aligned}
	\sum\limits_{\substack{p\le x\\ p\nmid M_{f}q  \\ \chi(p)=\pm 1}} \dfrac{\lambda_{f}(p)\overline{\lambda_{g}(p)}}{p}=O(1).
	\end{aligned}
	\end{equation}	
\item Suppose that there exists an odd prime $\ell\mid q$ such that $\ell\nmid M_f$. Then as $x\to\infty$, we have 
 	\begin{equation}\label{eq:estimation_lambdaf2_real}
 	\begin{aligned}
 	\sum\limits_{\substack{p\le x\\ p\nmid M_{f}q  \\ \chi(p)=\pm 1}} \dfrac{\lambda_{f}(p)\overline{\lambda_{f}(p)}}{p}=\dfrac{1}{2}\log\log x+O(1).
 	\end{aligned}
 	\end{equation}
Moreover, if $ g\neq f$ and $\ell^2\nmid M_g$, then as $ x\to\infty $, we have
	 \begin{equation}\label{eq:estimation_lambdafg_primitive}
	 \begin{aligned}
	 \sum\limits_{\substack{p\le x\\ p\nmid M_{f}q  \\ \chi(p)=\pm 1}} \dfrac{\lambda_{f}(p)\overline{\lambda_{g}(p)}}{p}=O(1).
	 \end{aligned}
	 \end{equation}

\end{enumerate}
	The implied constants in \eqref{eq:estimation_lambdaf_real}, \eqref{eq:estimation_lambdaf2_real} and \eqref{eq:estimation_lambdaffstar_real} depend on the form $ f $ and the character $ \chi $, and that in \eqref{eq:estimation_lambdafg_real} depends on the forms $ f,g $ and the character $ \chi $.

\end{lem}
\begin{proof}
(1) By Remark \ref{re:inducedprimitiveforms}, we have  $\chi(p)\lambda_{ f}(p)=\lambda_{f^{\star}}(p) $ for any prime $ p\nmid M_{f}q $. 
 	 For \eqref{eq:estimation_lambdaf_real}, by Lemma \ref{lem:lambdafg_primitive} \eqref{eq:lambdaf_primitive}, we have
 	 \begin{align*}
 	 2\sum\limits_{\substack{p\le x\\p\nmid M_{f}q \\ \chi(p)=\pm 1}}\dfrac{ \lambda_{f}(p)}{p}
 	 &=\sum\limits_{\substack{p\le x\\p\nmid M_{f}q}}\dfrac{ \lambda_{f}(p)}{p}\pm\sum\limits_{\substack{p\le x\\p\nmid M_{f}q}}\dfrac{\chi(p) \lambda_{f}(p)}{p}=\sum\limits_{\substack{p\le x\\p\nmid M_{f}q}}\dfrac{ \lambda_{f}(p)}{p}\pm\sum\limits_{\substack{p\le x\\p\nmid M_{f}q}}\dfrac{\lambda_{f^{\star}}(p)}{p}=O(1).
 	 \end{align*}
(2) It is not difficult to see the relation
	\begin{equation}
	\begin{aligned}\label{eq:estimate_twosums}
	2\sum\limits_{\substack{p\le x\\p\nmid M_{f}q \\ \chi(p)=\pm 1}}\dfrac{ \lambda_{f}(p)\overline{\lambda_{g}(p)}}{p}
	&=\sum\limits_{\substack{p\le x\\p\nmid M_{f}q}}\dfrac{ \lambda_{f}(p)\overline{\lambda_{g}(p)}}{p}\pm\sum\limits_{\substack{p\le x\\p\nmid M_{f}q}}\dfrac{\chi(p) \lambda_{f}(p)\overline{\lambda_{g}(p)}}{p}\\
	&=\sum\limits_{\substack{p\le x\\p\nmid M_{f}q}}\dfrac{ \lambda_{f}(p)\overline{\lambda_{g}(p)}}{p}\pm\sum\limits_{\substack{p\le x\\p\nmid M_{f}q}}\dfrac{\lambda_{f^{\star}}(p)\overline{\lambda_{g}(p)}}{p}
	\end{aligned} 
	\end{equation}
and the forms $ f $, $ f^{\star}=f_{\chi}^{\star}$ and $ g $ are primitive. By Lemma \ref{lem:lambdafg_primitive}, the first term in \eqref{eq:estimate_twosums} contributes $ \log\log x +O(1)$ if $g=f$ and $O(1)$ otherwise, while the second term contributes $\pm  \log \log x+O(1)$ if $g=f^*$ and $O(1)$ otherwise. 

For $g=f$ with $f$ not CM by $\chi$, the first term in \eqref{eq:estimate_twosums} thus contributes $\log\log x+O(1)$ and the second contributes $O(1)$, giving \eqref{eq:estimation_lambdaf2_realnonCM}.

For $g=f^{\star}$, the second term in \eqref{eq:estimate_twosums} always contributes $\pm \log \log x +O(1)$ and the first term contributes $\log \log x +O(1)$ if and only if $f$ has CM by $\chi$ (and $O(1)$ otherwise), giving \eqref{eq:estimation_lambdaffstar_real} and \eqref{eq:CM}. 

If $g\neq f$ and $g\neq f^{\star}$, then both terms in \eqref{eq:estimate_twosums} contribute $O(1)$, and we hence obtain \eqref{eq:estimation_lambdafg_real}.

\noindent
(3) By Lemma \ref{lem:f*}, $\ell\nmid M_f$ implies that $f\neq f_{\chi}^{\star}$. Since $f$ does not have CM by $\chi$, \eqref{eq:estimation_lambdaf2_realnonCM} implies  \eqref{eq:estimation_lambdaf2_real}.

Moreover, since $\ell^2\mid M_{f^{\star}}$ by Lemma \ref{lem:f*} and $\ell^2\nmid M_g$ in \eqref{eq:estimation_lambdafg_primitive}, we have $g\neq f^{\star}$, and hence \eqref{eq:estimation_lambdafg_primitive} follows immediately from \eqref{eq:estimation_lambdafg_real}.
\end{proof}

\section{Coefficients of arbitrary cusp forms and the proof of Theorem \ref{thm:signchange}}\label{sec:SignChangeThm}

Suppose that $k,L\in\N$ and $\psi$ is a character with conductor $L_{\psi}\mid L$. Writing 
\[
f\big|V_{\ell}(z):=f(\ell z),
\]
in \cite{atkin_hecke_1970}, A. Atkin and J. Lehner obtain the well-known decomposition
\begin{equation}\label{eq:decomposition}
\begin{aligned}
S_{k}(L,\psi)=\bigoplus_{\substack{M\mid L\\ L_{\psi}\mid M}}\bigoplus_{f\in H_{k}^{*}(M,\psi)}\Span_{\mathbb{C}}\left\{f\big|V_{\ell}: \ell \mid (L/M)\right\}.
	\end{aligned}
\end{equation}

\begin{lem}\label{lem:afestimation}
	Let $ k\ge 1  $ be an integer, $ N\ge 4 $ an integer divisible by $ 4 $ and $ \psi  $ a Dirichlet character modulo $ N $. Let $\chi$ be a primitive real character modulo $ q $ and $ \varepsilon\in\{\pm 1\} $. Suppose that $ \mathfrak{f}\in S_{k+1/2}^{*}(N,\psi) $ and $ t\ge 1 $ is a squarefree integer such that $ \mathfrak{a}_{\mathfrak{f}}(t)\not=0 $. Assume that the sequence $ \{\mathfrak{a}_{\mathfrak{f}}(tn^{2})\}_{n\in\mathbb{N}} $ is real. 
\begin{enumerate}[leftmargin=*,align=left,label={\rm(\arabic*)}]
\item
Then as $ x\to \infty $, we have 
	\begin{equation}\label{eq:mathfrakafestimation}
	\sum\limits_{\substack{p\le x\\ p\nmid q  \\ \chi(p)=\pm 1}} \dfrac{\mathfrak{a}_{\mathfrak{f}}(tp^{2})}{p^{k+1/2}}=O_{\mathfrak{f},t,\chi}(1).
	\end{equation}
\item	If there exists an odd prime $r\mid q$ such that $r\nmid N$, then as $ x\to \infty $, we have
	\begin{equation}\label{eq:mathfrakafsquare}
	\sum\limits_{\substack{p\le x\\ p\nmid q\\ \chi(p)=\varepsilon}} \dfrac{\mathfrak{a}_{\mathfrak{f}}(tp^{2})^{2}}{p^{2k}}=C\log\log x+O(1)
	\end{equation}	
	holds for some $C>0$ for both $ \varepsilon=1 $ and $ \varepsilon=-1 $. 
\item Suppose that every prime divisor of $q$ divides $N$. The equality \eqref{eq:mathfrakafsquare} holds with $C>0$ unless the $ t $-th Shimura correspondence $ f_{t} $ satisfies
\begin{equation}\label{eqn:ftbad}
f_t\otimes \chi_{N^2} = \sum_{i} c_i\left(f_i-\varepsilon f_i^{\star}\right)\otimes \chi_{N^2},
\end{equation}
where $f_i$ run through all of the primitive forms of level dividing $N$ and $c_i\in\C$. If \eqref{eqn:ftbad} holds, then $C=0$ and moreover $a_{f_t}(p)=0$ for every prime $p\nmid N$ with $\chi(p)=\varepsilon$.

\end{enumerate}

The implied constants $ C$ and those occurring in the $ O $-symbols depend on $ \mathfrak{f} $, $ t $, $ \chi $ and $ \varepsilon $.
\end{lem}
\begin{proof}
	 Applying the M\"obius inversion formula to \eqref{eq:relationaft_af}, we have
\[
\mathfrak{a}_{\mathfrak{f}}(tn^{2})=\sum\limits_{d\mid n}\mu(d)\psi_{t,N}(d)d^{k-1}a_{f_{t}}\left(\dfrac{n}{d}\right),
\]
	 where $ a_{f_{t}}(n) $ is the $ n $-th coefficient of $ f_{t} $. Write $ a_{f_{t}}(n)=\lambda_{f_{t}}(n)n^{k-1/2} $. Then we may rewrite the above formula as 
	 

\[ 	
\dfrac{\mathfrak{a}_{\mathfrak{f}}(tn^{2})}{n^{k-1/2}}=\sum\limits_{d\mid n}\dfrac{\mu (d)\psi_{t,N}(d)}{\sqrt{d}}\lambda_{f_{t}}\left(\dfrac{n}{d}\right).
	 \]
	 Considering the special case that $ n=p $ is a prime and noting that $ \lambda_{f_{t}}(1)=\mathfrak{a}_{\mathfrak{f}}(t) $ yields
	 \begin{equation}\label{eq:af_lambdaft}
	 \dfrac{\mathfrak{a}_{\mathfrak{f}}(tp^{2})}{p^{k-1/2}}=\lambda_{f_{t}}(p)-\dfrac{\psi_{t,N}(p)}{\sqrt{p}}\mathfrak{a}_{\mathfrak{f}}(t).
	 \end{equation}
	 Applying the decomposition \eqref{eq:decomposition} to $ S_{2k}(N/2,\psi^{2}) $, we obtain a basis
	 \begin{align*}
	 \bigcup\limits_{\substack{M\mid (N/2)\\ L_{\psi^2}\mid M}}\left\{f\big|V_{\ell}:\ell\mid \dfrac{N/2}{M} ,f\in H_{2k}^{*}(M,\psi^{2})\right\}.
	 \end{align*}
 Hence $ f_t\in S_{2k}(N/2,\psi^{2}) $ can be written as 
\begin{equation}\label{eq:fdecomp}
	 f_t(z)=\sum\limits_{i}\sum\limits_{\ell\mid (N/(2M_{i}))}c_{i,\ell}f_{i}(\ell z),
\end{equation}
	 where $ f_{i}\in H_{2k}^{*}(M_{i},(\psi^{2})_{M_{i}}) $ is primitive of level $ M_{i} $ and the $ c_{i,\ell} $ are scalars depending on $ f $.
	 
	 
	 For any prime $ p\nmid Nq $, the terms with $\ell\neq 1$ do not contribute anything to the $p$-th Fourier coefficient, so, comparing coefficients of the functions on each side of \eqref{eq:fdecomp}, we see that
	\begin{align*}
	\lambda_{f_{t}}(p)=\sum\limits_{i}c_{i}\lambda_{f_{i}}(p),
	\end{align*}
	where $ c_{i}:=c_{i,1} $. Moreover, since 
\[
\sum\limits_{i}c_{i}=\lambda_{f_t}(1)=\mathfrak{a}_{\mathfrak{f}}(t)\neq 0
\]
by assumption, not all the $ c_{i} $ are zero. Expressing $ \lambda_{f_{t}} $ by the linear combination of $ \lambda_{f_{i}} $ in \eqref{eq:af_lambdaft}, we have
	\begin{equation}\label{eq:af_lambdafi}
	\dfrac{\mathfrak{a}_{\mathfrak{f}}(tp^{2})}{p^{k-1/2}}=\sum\limits_{i}c_{i}\lambda_{f_{i}}(p)-\dfrac{\psi_{t,N}(p)}{\sqrt{p}}\mathfrak{a}_{\mathfrak{f}}(t).
	\end{equation}
\noindent

\noindent 
(1)
	Dividing \eqref{eq:af_lambdafi} by $ p $ and summing over $ p\le x $ with $ \chi(p)=\pm 1 $ but $ p\nmid Nq $ on \eqref{eq:af_lambdafi} and then applying Lemma \ref{lem:lambda_estimation_real} \eqref{eq:estimation_lambdaf_real} to $ f_{i} $ in \eqref{eq:af_lambdafi}, we deduce that
	\begin{align*}
	\sum\limits_{\substack{p\le x\\ p\nmid Nq\\ \chi(p)=\pm 1}}\dfrac{\mathfrak{a}_{\mathfrak{f}}(tp^{2})}{p^{k+1/2}}=\sum\limits_{i}c_{i}\sum\limits_{\substack{p\le x\\p\nmid Nq\\ \chi(p)=\pm 1}}\dfrac{\lambda_{f_{i}}(p)}{p}-\mathfrak{a}_{\mathfrak{f}}(t)\sum\limits_{\substack{p\le x\\p\nmid Nq\\ \chi(p)=\pm 1}}\dfrac{\psi_{t,N}(p)}{p^{3/2}}=O_{\mathfrak{f},t,\chi}(1),
	\end{align*}
	thereby showing \eqref{eq:mathfrakafestimation}.
\noindent

\noindent
(2) Since $\mathfrak{a}_{\mathfrak{f}}(tn^{2})\in \R$, multiplying \eqref{eq:af_lambdafi} by it its complex conjugate yields
	\begin{multline}\label{eq:afsquare}
	\dfrac{\mathfrak{a}_{\mathfrak{f}}(tp^{2})^{2}}{p^{2k-1}}=\sum\limits_{i}|c_{i}|^{2}|\lambda_{f_{i}}(p)|^{2}+\sum\limits_{i\not=j}c_{i}\overline{c_{j}}\lambda_{f_{i}}(p)\overline{\lambda_{f_{j}}(p)} \\
	+\mathfrak{a}_{\mathfrak{f}}(t)^{2}\dfrac{|\psi_{t,N}(p)|^{2}}{p}-2\Re \left( \sum\limits_{i}\dfrac{\overline{c_{i}\lambda_{f_{i}}(p)}\psi_{t,N}(p)}{\sqrt{p}}\mathfrak{a}_{\mathfrak{f}}(t)\right).
	\end{multline}
 Dividing \eqref{eq:afsquare} by $ p $ and summing over $p\leq x$ with $\chi(p)=\pm 1$ but $ p\nmid Nq $ yields
	\begin{multline}\label{eqn:sumabssquared}
	\sum\limits_{\substack{p\le x\\p\nmid Nq\\ \chi(p)=\pm 1}}\dfrac{\mathfrak{a}_{\mathfrak{f}}(tp^{2})^{2}}{p^{2k}}=\sum\limits_{i}|c_{i}|^{2}\sum\limits_{\substack{p\le x\\p\nmid Nq\\ \chi(p)=\pm 1}}\dfrac{|\lambda_{f_{i}}(p)|^{2}}{p}+\sum\limits_{i\not=j}c_{i}\overline{c_{j}}\sum\limits_{\substack{p\le x\\p\nmid Nq\\ \chi(p)=\pm 1}}\dfrac{\lambda_{f_{i}}(p)\overline{\lambda_{f_{j}}(p)}}{p}\\
	+\mathfrak{a}_{\mathfrak{f}}(t)^{2}\sum\limits_{\substack{p\le x\\p\nmid Nq\\ \chi(p)=\pm 1}}\dfrac{|\psi_{t,N}(p)|^{2}}{p^{2}}-2 \sum\limits_{\substack{p\le x\\p\nmid Nq\\ \chi(p)=\pm 1}}\Re \left( \sum\limits_{i}\dfrac{\overline{c_{i}\lambda_{f_{i}}(p)}\psi_{t,N}(p)}{p^{3/2}}\mathfrak{a}_{\mathfrak{f}}(t)\right).
	\end{multline}	
The last two terms are $ O_{\mathfrak{f},t,\chi}(1) $ from the fact that the sum $ \sum_{i}|\lambda_{f_{i}}(p)| $ is bounded \cite[see Corollary 5.2]{iwaniec_topic_1997}. Now consider the sum of the first and second terms. 
	
	Define the index sets $
	I^{\star}:=\{i: f_{i}\;\text{has CM by}\;\chi\}$ and $
	J:=\{(i,j): i\not=j\;\text{and}\; f_{j}=f_{i}^{\star}\}$.
Write 
	\begin{align*}
		S_{1}:=\sum\limits_{i\not\in I^{\star}}|c_{i}|^{2} \qquad S_{1}^{\star}:=\sum\limits_{i\in I^{\star}}|c_{i}|^{2} \qquad
		S_{2}:=\sum\limits_{(i,j)\in J}c_{i}\overline{c_{j}}.
	\end{align*}
Clearly, not both $ S_{1} $ and $ S_{1}^{\star} $ are zero. For the first term in \eqref{eqn:sumabssquared}, we use \eqref{eq:estimation_lambdaf2_realnonCM} and \eqref{eq:CM} from Lemma \ref{lem:lambda_estimation_real} to obtain
	\begin{equation}\label{eq:estimation_fi2}
	\begin{aligned}
	\sum\limits_{i}|c_{i}|^{2}\sum\limits_{\substack{p\le x\\p\nmid Nq\\ \chi(p)=\varepsilon}}\dfrac{|\lambda_{f_{i}}(p)|^{2}}{p}=\dfrac{S_{1}+(\varepsilon+1) S_{1}^{\star}}{2}\log\log x+O_{\mathfrak{f},t,\chi}(1).
	\end{aligned}
	\end{equation}
For the second term in \eqref{eqn:sumabssquared}, since $ f_{i}\not=f_{j} $ for $ i\not=j $, \eqref{eq:estimation_lambdaffstar_real} implies
	\begin{equation}\label{eq:estimation_fifj}
	\begin{aligned}
	\sum\limits_{(i,j)\in J}c_{i}\overline{c_{j}}\sum\limits_{\substack{p\le x\\p\nmid Nq\\ \chi(p)=\varepsilon}}\dfrac{\lambda_{f_{i}}(p)\overline{\lambda_{f_{j}}(p)}}{p}=\dfrac{\varepsilon S_{2}}{2}\log\log x+O_{\mathfrak{f},t,\chi}(1).
	\end{aligned}
	\end{equation}
If $i\neq j$ and $(i,j)\notin J$, then $f_j\neq f_i$ and $f_j\neq f_i^{\star}$ (if $f_i$ has CM by $\chi$, then $f_i^{\star}=f_i$, so $f_j\neq f_i$ implies that $f_j\neq f_i^{\star}$). Hence for the remaining terms in \eqref{eqn:sumabssquared} we may use \eqref{eq:estimation_lambdafg_real} to obtain
	\begin{equation}\label{eq:estimation_fifj2}
	\begin{aligned}
	\sum\limits_{\substack{i\neq j\\ (i,j)\notin J}}c_{i}\overline{c_{j}}\sum\limits_{\substack{p\le x\\p\nmid Nq\\ \chi(p)=\varepsilon}}\dfrac{\lambda_{f_{i}}(p)\overline{\lambda_{f_{j}}(p)}}{p}=O_{\mathfrak{f},t,\chi}(1).
	\end{aligned}
	\end{equation}
Combining \eqref{eq:estimation_fi2}, \eqref{eq:estimation_fifj}, and \eqref{eq:estimation_fifj2}, the sum of the first and second terms in \eqref{eqn:sumabssquared} is given by
	\begin{equation}\label{eq:sum_firstsecond}
	\dfrac{S_{1}+\varepsilon S_{2}+(\varepsilon+1)S_{1}^{\star}}{2}\log\log x+O_{\mathfrak{f},t,\chi}(1).
	\end{equation}
	If there exists an odd prime $r\mid q$ such that $r\nmid N$, then Lemma \ref{lem:f*} implies that $r^2\mid M_{f_i^{\star}}$. Since $r^2\nmid N$ and $M_{j}\mid N$ for every $j$ (including $j=i$), $f_{i}^{\star}\neq f_j$ for every $j$ and hence $ I^{\star}=J=\emptyset$ and $S_2= S_{1}^{\star}=0 $ in this case. Since $S_1+S_1^*>0$, we furthermore obtain that $ S_{1}>0 $. Hence \eqref{eq:sum_firstsecond} becomes 
\[
\frac{S_{1}}{2}\log\log x+O_{\mathfrak{f},t,\chi}(1)
\]
with $S_1>0$. This yields the claim for the case that such an odd prime $\ell$ exists.
\noindent

\noindent
(3) We write 
\[
S_1+\varepsilon S_2+(\varepsilon+1)S_1^{\star}=:Ce^{i\theta}
\]
with $C\geq 0$ and $-\pi<\theta\leq \pi$ and note that since the left-hand side of \eqref{eq:mathfrakafsquare} is a sum of nonnegative real numbers, if $C\neq 0$ then we must have $\theta=0$ (otherwise the limit in \eqref{eq:sum_firstsecond} would diverge to $+e^{i\theta}\infty$ as $x\to\infty$ while for each $x$ it equals a nonnegative real number, a contradiction).

Writing $a_i:=\Re(c_i)$ and $b_i:=\im(c_i)$, we conclude that 
\[
S_1+\varepsilon S_2+(\varepsilon+1)S_1^{\star}=\Re\left(S_1+\varepsilon S_2+(\varepsilon+1)S_1^{\star}\right)
\]
 and hence
\begin{equation}\label{eqn:Ceval1}
C=S_1+\varepsilon S_2+(\varepsilon+1)S_1^{\star}=\sum_{i\notin I^{\star}} \left(a_i^2+b_i^2\right) +\varepsilon  \sum_{(i,j)\in J} \left(a_ia_j + b_ib_j\right) + \left(\varepsilon+1\right) \sum_{i\in I^{\star}}\left(a_i^2+b_i^2\right).
\end{equation}
Consider the set 
\[
I_J:=\{ i: \exists j\;\text{ s.t. } (i,j)\in J\}
\]
and note that if $(i,j)\in J$ then $(j,i)\in J$, but since $\left(f_i^{\star}\right)^{\star}=f_i$, the tuples in $J$ appear in pairs, and hence there does not exist $j'\neq j$ such that $(i,j')\in J$. Thus \eqref{eqn:Ceval1} becomes
\begin{align}
\nonumber
C=\sum_{i\notin \left(I^{\star}\cup I_J\right)} \left(a_i^2+b_i^2\right) + \sum_{(i,j)\in J} \frac{1}{2}\left(a_i^2+b_i^2+a_j^2+b_j^2\right)+\varepsilon \left(a_ia_j + b_ib_j\right) + \left(\varepsilon+1\right) \sum_{i\in I^{\star}}\left(a_i^2+b_i^2\right)\\
\label{eqn:Ceval2}=\sum_{i\notin \left(I^{\star}\cup I_J\right)} \left(a_i^2+b_i^2\right) + \frac{1}{2}\sum_{(i,j)\in J} \left(\left(a_i+\varepsilon a_j\right)^2 +\left(b_i+\varepsilon b_j\right)^2\right)+ \left(\varepsilon+1\right) \sum_{i\in I^{\star}}\left(a_i^2+b_i^2\right).
\end{align}
Hence we conclude that $C>0$ unless all of the following hold:
\begin{itemize}
\item If $i\notin \left(I^{\star}\cup I_J\right)$, then $c_i=0$. 
\item If $(i,j)\in J$, then  $c_i=-\varepsilon c_j$.
\item If $\varepsilon=1$, then $c_i=0$ for every $i\in I^{\star}$.
\end{itemize}
Noting that, since $\chi_{N^2}$ annihilates $f_i\big|V_{\ell}$ for every $\ell>1$, 
\[
f_t\otimes \chi_{N^2} = \sum_{i} c_i f_i\otimes \chi_{N^2}
\]
and writing $f_i=\frac{1}{2}\left(f_i+f_i^{\star}\right)$ for $i\in I^{\star}$, these three conditions are equivalent to \eqref{eqn:ftbad}. 

Finally note that if \eqref{eqn:ftbad} holds, then for $p\nmid N$ with $\chi(p)=\varepsilon$ we have 
\[
a_{f_t}(p)= \sum_{i} c_i\left(a_{f_i}(p)-\varepsilon a_{f_i^{\star}}(p)\right)=\sum_{i} c_i\left(a_{f_i}(p)-\varepsilon \chi(p)a_{f_i}(p)\right)=\sum_{i} c_i\left(a_{f_i}(p)-a_{f_i}(p)\right)=0,
\]
where we used the fact that 
\[
a_{f_i^{\star}}(p)=a_{f_i\otimes\chi}(p)=\chi(p)a_{f_i}(p).
\] 
\end{proof}

\begin{proof}[Proof of Theorem \ref{thm:signchange}]
	We claim that if 
	\begin{equation}\label{eq:aspt_aftp22}
	\sum\limits_{\substack{p\le x\\p\nmid Nq\\ \chi(p)=\varepsilon}}\dfrac{\mathfrak{a}_{\mathfrak{f}}(tp^{2})^{2}}{p^{2k}}=C_{\mathfrak{f},t,\chi,\varepsilon}\log\log x+O_{\mathfrak{f},t,\chi,\varepsilon}(1),
	\end{equation} 
	for some $ C_{\mathfrak{f},t,\chi,\varepsilon}>0 $, then the assertion is true for $ \chi(p)=\varepsilon $.
	By Deligne's bound \cite{Deligne1974}, $ |\lambda_{f_{i}}(p)|\le 2$ (as $ |a_{f_{i}}(p)|\le 2p^{k-1/2} $). Then \eqref{eq:af_lambdafi} implies
	\[
	|\mathfrak{a}_{\mathfrak{f}}(tp^{2})p^{-(k-1/2)}|\le 2\sum\limits_{i}|c_{i}|+|\mathfrak{a}_{\mathfrak{f}}(t)|=:C_{\mathfrak{f},t}.
	\]	
	Suppose that $ \mathfrak{a}_\mathfrak{f}(tp^{2}) $ are of the same sign for $ y\le p\le x $ with $ p\nmid Nq $ and $ \chi(p)=\varepsilon $. Without loss of generality, assume that $\mathfrak{a}_\mathfrak{f}(tp^{2}) >0$. Then 
	\begin{equation}\label{eq:comparision}
	\sum\limits_{\substack{y\le p\le x\\p\nmid Nq\\ \chi(p)=\varepsilon}}\dfrac{\mathfrak{a}_{\mathfrak{f}}(tp^{2})^{2}}{p^{2k}}\le C_{\mathfrak{f},t}\sum\limits_{\substack{y\le p\le x\\p\nmid Nq\\ \chi(p)=\varepsilon}}	\dfrac{\mathfrak{a}_{\mathfrak{f}}(tp^{2})}{p^{k+1/2}}.
	\end{equation}
	The left-hand side of \eqref{eq:comparision} is given by
	\[ 	
C_{\mathfrak{f},t,\chi,\varepsilon}\log \left(\dfrac{\log x}{\log y}\right)+O_{\mathfrak{f},t,\chi,\varepsilon}(1) \]
	from the assumption \eqref{eq:aspt_aftp22}. However, the right-hand side is $ O_{\mathfrak{f},t,\chi}(1) $ from Lemma \ref{lem:afestimation}  \eqref{eq:mathfrakafestimation} for all $ x\ge y\ge 2 $. It is impossible if $ y=x^{\delta} $ with a sufficiently small constant $ \delta=\delta(\mathfrak{f},t,\chi,\varepsilon)>0 $. The claim is proved. Combining the claim with \eqref{eq:mathfrakafsquare} from Lemma \ref{lem:afestimation}, we are done.
\end{proof}

\section{Spinor genera and the proof of Theorem \ref{thm:signchangedivisor}}\label{sec:signchangedivisor}

In this section, we investigate the case when every odd prime dividing the conductor of $\chi$ also divides the level of the modular form. To prove Theorem \ref{thm:signchangedivisor} (1), we require the following lemma \cite[Lemma 7.5 and 7.6]{apostol_introduction_1986} for the case when $ \psi_{t,N} $ is a real character.

\begin{lem}\label{lem:Dirichlet}
	Let $ \chi $ be a nonprincipal character. As $ x\to \infty $, we have
	\begin{align*}
	\sum\limits_{p\le x}\dfrac{\chi(p)\log p}{p}&=-L^{\prime}(1,\chi)\sum\limits_{n\le x}\dfrac{\mu(n)\chi(n)}{n}+O(1),\\	L(1,\chi)\sum\limits_{n\le x}\dfrac{\mu(n)\chi(n)}{n}&=O(1),
	\end{align*}
	where $ L(s,\chi)=\sum_{n=1}^{\infty}\chi(n)/n^{s} $ is the Dirichlet L-function.
\end{lem}

Now, we begin by showing Theorem \ref{thm:signchangedivisor} (1).

\begin{proof}[Proof of Theorem \ref{thm:signchangedivisor} (1)]
As in the proof of Theorem \ref{thm:signchange}, the form exhibits sign changes whenever \eqref{eq:aspt_aftp22} holds. Hence for $\mathfrak{f}$ such that \eqref{eqn:ftnosignchange} does not hold, Lemma \ref{lem:afestimation} implies that \eqref{eqn:primeDseq} exhibits sign changes. 

In the case that \eqref{eqn:ftnosignchange} holds, Lemma \ref{lem:afestimation} implies that $a_{f_t}(p)=0$ for every $p\nmid N$ with $\chi_D(p)=\varepsilon$. Then \eqref{eq:Shimuraliftcoeff} implies 
\[
0=a_{f_t}(p)=\psi_{t,N}(p) p^{k-1}\mathfrak{a}_{\mathfrak{f}}(t)+\mathfrak{a}_{\mathfrak{f}}\left(tp^2\right).
\]
Thus 
\[
\mathfrak{a}_{\mathfrak{f}}\left(tp^2\right)=-\psi_{t,N}(p)p^{k-1}\mathfrak{a}_{\mathfrak{f}}(t).
\]
Note that $ \mathfrak{a}_{\mathfrak{f}}(n)  $ are all real and hence $ \psi_{t,N}(p) $ is real for any prime $ p\nmid N $ with $ \chi_{D}(p)=\varepsilon $. 

 Let $ m_{1},m_{2} $ be the smallest positive periods of the characters $ \psi_{t,N} $ and $ \chi_{D} $, respectively, and $ m:=\lcm(m_{1},m_{2}) $. If $\psi_{t,N}\simeq\chi_{D}^{j}$, then $-\psi_{t,N}(p)=-\varepsilon^{j}$ for any prime $ p $ not dividing $ m $ and so for any prime $p$ with $\chi_{D}(p)=\varepsilon$ (and so $ p\nmid m_{1} $) but $ p\nmid m_{2} $. Hence we have
 \begin{align*}
 \mathfrak{a}_{\mathfrak{f}}\left(tp^2\right)=
 \begin{cases}
 0&\text{if $ \chi_{D}(p)=\varepsilon $ and $ p\mid m_{2} $}, \\
	-\varepsilon^{j} p^{k-1} \mathfrak{a}_{\mathfrak{f}}(t) &\text{if $ \chi_{D}(p)=\varepsilon $ and $ p\nmid m_{2} $}. \\
 \end{cases}
 \end{align*} 
 Thus there are no sign changes in this case.

If $ \psi_{t,N}\not\simeq\chi_{D}^{j} $ for $ j=0,1 $, suppose that the subsequence does not exhibit sign changes. Then one of the following holds:
\begin{itemize}
\item[(a)] $ \psi_{t,N}(p)=\chi_{D}(p)=\varepsilon $ for any prime $ p $ with $ \chi_{D}(p)=\varepsilon $; 

\item[(b)] $ \psi_{t,N}(p)=-\chi_{D}(p)=-\varepsilon $ for any prime $ p $ with $ \chi_{D}(p)=\varepsilon $.
\end{itemize}
\textbf{Case I:} either (a) or (b) holds with $ \varepsilon=-1 $:

Choose a prime $ p_{0} $ with $ \chi_{D}(p_{0})=-1 $ and $ p\nmid m_{2} $. For any prime $ p $ with $ \chi_{D}(p)=1 $ and $ p\nmid m_{2} $, we have $ \chi_{D}(pp_{0})=-1 $. Since $ \gcd(m,pp_{0})=1 $, by Dirichlet's theorem on arithmetic progressions, there exists some prime $ q $ such that $ q\equiv pp_{0}\pmod{m} $. Hence $ \chi_{D}(q)=-1$. If the statement (a) holds, then we have
\begin{align*}
\psi_{t,N}(pp_{0})=\psi_{t,N}(q)=-1=\chi_{D}(q)=\chi_{D}(pp_{0})
\end{align*}
and $ \psi_{t,N}(p_{0})=\chi_{D}(p_{0})=-1 $; if the statement (b) holds, then we have
\begin{align*}
\psi_{t,N}(pp_{0})=\psi_{t,N}(q)=1=-\chi_{D}(q)=-\chi_{D}(pp_{0})
\end{align*}
and $ \psi_{t,N}(p_{0})=-\chi_{D}(p_{0})=1 $. In both cases, we conclude that $ \psi_{t,N}(p)=\chi_{D}(p) $ for any prime $ p $ with $ \chi_{D}(p)=1 $ and $ p\nmid m_{2} $, and so for any prime $ p $ with $ p\nmid m_{2} $. Hence $ \psi_{t,N}(p)=\chi_{D}(p) $ for any prime $ p $ not dividing $ m_{1}m_{2} $. Therefore, $ \psi_{t,N}\simeq\chi_{D} $, yielding a contradiction.

\noindent
\textbf{Case II:}  (b) holds with $ \varepsilon=1 $:

By Dirichlet's theorem, there exists a prime $ q $ with $ q\equiv 1\pmod{m} $. Then $ \chi_{D}(q)=1 $. By (b), we have $
\psi_{t,N}(q)=-\chi_{D}(q)=-1 $.
But $ \psi_{t,N}(q)=\psi_{t,N}(1)=1 $ from $ q\equiv 1\pmod{m} $. This is a contradiction.

\noindent
\textbf{Case III:}  (a) holds with $ \varepsilon=1 $:

For any prime $ p $ with $ \chi_{D}(p)\not=0$ and $ p\nmid m $, since $ \gcd(p^{2},m)=1 $, Dirichlet's theorem implies that there must exist a prime $ q\equiv p^{2}\pmod{m} $ and thus $ \chi_{D}(q)=\chi_{D}(p^{2})=1 $. By (a), we deduce that
\begin{align*}
\psi_{t,N}(p^{2})=\psi_{t,N}(q)=\chi_{D}(q)=1.
\end{align*}
Therefore, $ \psi_{t,N}(p)\in\mathbb{R} $ for any prime with $ \chi_{D}(p)\not=0 $ and $ p\nmid m $. Note that $ \chi_{D}(p)=0 $ implies $ \psi_{t,N}(p)=0 $. It follows that $ \psi_{t,N}(p)\in\mathbb{R} $ for any prime $ p $. Hence $ \psi_{t,N} $ is a real character.

From the above discussion, if the subsequence does not exhibit sign changes, then $ \psi_{t,N} $ must be real. We now suppose that $ \psi_{t,N} $ is real and $\psi_{t,N}\not\simeq \chi_D^j$, and will prove that the sign changes do indeed occur. From $
\mathfrak{a}_{\mathfrak{f}}\left(tp^2\right)=-\psi_{t,N}(p)p^{k-1} \mathfrak{a}_{\mathfrak{f}}(t)
$, it is sufficient to show that the following both occur:
\begin{itemize}
\item[(c)] there exist infinitely many primes $ p $ with $ \chi_{D}(p)=\varepsilon $  such that $ \psi_{t,N}(p)=\chi_{D}(p)=\varepsilon $;
\item[(d)] there exist infinitely many primes $ p $ with $ \chi_{D}(p)=\varepsilon $ such that $\psi_{t,N}(p)=-\chi_{D}(p)=-\varepsilon $.
\end{itemize}
By analytic number theory it suffices to show that for $ \varepsilon=\pm 1 $ the sum
\begin{align*}
\sum\limits_{\substack{p\le x\\\chi_{D}(p)=\varepsilon\,,\,\psi_{t,N}(p)=\pm\varepsilon}}\dfrac{\log p}{p}
\end{align*}
diverges as $x\to\infty$. Without loss of generality, we only consider the case $ \chi_{D}(p)=1 $ and $ \psi_{t,N}(p)=-1 $. Then one can check that
\begin{multline*}
\sum\limits_{\substack{p\le x\\\chi_{D}(p)=1\,,\,\psi_{t,N}(p)=-1}}\dfrac{\log p}{p}=\dfrac{1}{2}\left(\sum\limits_{\substack{p\le x\\\psi_{t,N}(p)=-1}}\dfrac{\log p}{p}+	\sum\limits_{\substack{p\le x\\\psi_{t,N}(p)=-1}}\dfrac{\chi_{D}(p)\log p}{p}\right)+O(1)  \\
=\dfrac{1}{4}\left(\sum\limits_{p\le x}\dfrac{\log p}{p}-\sum\limits_{p\le x}\dfrac{\psi_{t,N}(p)\log p}{p}+\sum\limits_{p\le x}\dfrac{\chi_{D}(p)\log p}{p} -\sum\limits_{p\le x}\dfrac{\psi_{t,N}(p)\chi_{D}(p)\log p}{p}\right)+O(1).
\end{multline*}
Since $ \chi_{D}$, $\psi_{t,N}$, and $\psi_{t,N}\chi_{D} $ are nonprincipal and real (from the assumptions that $\psi_{t,N}\not\simeq \chi_D^j$ and $D\neq 1$), the associated $L$-functions do not have poles at $ s=1 $. So the last three terms contribute $ O(1) $ by Lemma \ref{lem:Dirichlet} and the first term contributes $ \log x+O(1) $  by the Prime Number Theorem. Hence
\begin{align*}
\sum\limits_{\substack{p\le x\\\chi_{D}(p)=1\,,\,\psi_{t,N}(p)=-1}}\dfrac{\log p}{p}=\dfrac{1}{4}\log x+O(1),
\end{align*}
and the claim is proved. Thus \eqref{eqn:primeDseq} exhibits infinitely many sign changes when $ \psi_{t,N}\not\simeq\chi_{D}^{j} $ for $ j=0,1 $.

Finally note that \eqref{eqn:ftnosignchange} cannot hold for both $\varepsilon$ and $-\varepsilon$ because that would contradict the assumption that $a_{f_t}(1)=\lambda_{f_t}(1)=\mathfrak{a}_{\mathfrak{f}}(t)\neq 0$. Moreover, the condition $\psi_{t,N}\simeq \chi_D^j$ can occur for at most one choice of $t$ and since $D\mid N$ and $\psi$ is a character modulo $N$, we have $t\mid N$. 
\end{proof}
While Theorem \ref{thm:signchangedivisor} (1) should yield many examples where sign changes are not exhibited, it is not entirely obvious that the conditions $\psi_{t,N}\simeq \chi_D^j$, $\mathfrak{a}_{\mathfrak{f}}(t)\neq 0$, and \eqref{eqn:ftnosignchange} are simultaneously satisfied.  We hence construct an explicit example where we can verify all three conditions. The construction goes through the theory of spinor genera of ternary quadratic forms. 

\begin{proof}[Proof of Theorem \ref{thm:signchangedivisor} (2)]
Consider the quadratic forms 
\begin{align*}
Q_{1}(x,y,z)&:=x^{2}+48y^{2}+144z^{2}, \\
Q_{2}(x,y,z)&:=4x^{2}+48y^{2}+49z^{2}+4xz+48yz,
\end{align*}
discussed for example in \cite[(4.18), p.\hskip 0.1cm 9]{baeza_representation_2004} and \cite[\S 7.3 An example]{Hanke_2004}.

The forms $Q_1$ and $Q_2$ are in the same spinor genus and are moreover representatives for the only two classes in their spinor genus.

The level $ N_{Q_j} $ and the determinant $ D_{Q_j} $ of $ Q_{j} $ ($ j=1,2 $) are $ 2^{6}3^{2} $ and $ 2^{8}3^{3} $, respectively. Schulze-Pillot found that $ t=1 $ is a primitive spinor exception and is represented by $ Q_{1} $, but not represented by $ Q_{2} $. By the theory of spinor genera, the spinor genus does not primitively represent the integers $tp^2$, or in other words, since all representations of $tp^2$ come from representations of $t$,
\[
r(tp^2,Q_j)=r(t,Q_j).
\]
Thus in particular
\[
r(tp^{2},Q_{2})=r(t,Q_2)=0
\]
for any odd prime $ p\equiv -1\pmod{3} $, i.e. $(-3/p)=-1$.

Plugging in the expansion \eqref{eqn:thetaQexpand} and noting that $r(t,Q)$ is the $t$-th coefficient of $\theta_Q$, \eqref{eqn:genspnexpand} implies that
\[
0=r(t,Q_{2})=a_{E}(t,Q_{2})+a_{H}(t,Q_{2})+a_{f}(t,Q_{2})=r(t,\spn(Q_{2}))+a_{f}(t,Q_{2}).\]
Therefore, we have 
\[
 a_{f}(t,Q_{2})=-r(t,\spn(Q_{2}))\not=0.
\]
For any inert prime $ p $ in $ \mathbb{Q}(\sqrt{-3}) $, i.e., $ (-3/p)=-1 $,  replacing $ t $ by $ tp^{2} $, we have $ a_{f}(tp^{2},Q_{2})=-r(tp^{2},\spn(Q_{2})) $ analogously from $ r(tp^{2},Q_{2})=0 $. Since $ t $ is represented by $ \spn(Q_{2}) $, $ tp^{2} $ is not primitivley represented by $ \spn(Q_{2}) $. Therefore, \[ r(tp^{2},\spn(Q_{2}))-r(t,\spn(Q_{2}))=r^{*}(tp^{2},\spn(Q_{2}))=0 ,\] where $ r^{*}(n,\spn(Q)) $ denotes the number of primitive representation of $ n $ by $ Q $. Namely, $ r(tp^{2},\spn(Q_{2}))=r(t,\spn(Q_{2})) $. It follows that 
\[	a_{f}(t,Q_{2})=-r(t,\spn(Q_{2}))=-r(tp^{2},\spn(Q_{2}))=a_{f}(tp^{2},Q_{2}).
\]
Hence we see that $ a_{f}(tp^{2},Q_{2}) $ has the same sign for $ t=1 $ and any prime $ p\equiv -1\pmod{3} $. 
\end{proof}

\end{document}